\documentclass[11pt]{amsart}

\usepackage{amssymb}
\usepackage{amsmath}
\usepackage{amsthm}
\usepackage{amscd}
\usepackage{amsfonts}
\usepackage{ascmac}
\usepackage[usenames]{color}
\usepackage{graphics}

\theoremstyle{plain}
\newtheorem{thm}{Theorem}[section]
\newtheorem{prop}[thm]{Proposition}
\newtheorem{lem}[thm]{Lemma}

\newtheorem{ex}[thm]{Example}
\newtheorem{rem}[thm]{Remark}

\newcommand{\bfC}{{\mathbf C}}
\newcommand{\bfP}{{\mathbf P}}
\newcommand{\bfR}{{\mathbf R}}

\newcommand{\bari}{{\overline i}}
\newcommand{\barj}{{\overline j}}
\newcommand{\bark}{{\overline k}}

\newcommand{\baru}{{\overline u}}

\newcommand{\barz}{{\overline z}}

\newcommand{\barbet}{{\overline \beta}}

\newcommand{\barL}{\overline {L}}

\newcommand{\barpartial}{{\overline \partial}}

\newcommand{\mapright}[1]{\smash{\mathop{   \hbox to 0.7cm{\rightarrowfill}}
  \limits^{#1}}}

\def\grad{\mathrm{grad}}

\renewcommand{\emph}[1]{{\color{red} \it #1}}

\definecolor{orange}{cmyk}{0, 0.7, 1, 0}
\definecolor{light-green}{cmyk}{0.5, 0, 0.5, 0}
\definecolor{light-blue}{cmyk}{0.5, 0, 0, 0}
\definecolor{light-yellow}{cmyk}{0,0,0.6,0}
\definecolor{dark-green}{cmyk}{0.7, 0, 0.7, 0.5}

\title{Conformally Einstein-Maxwell K\"ahler metrics and structure of the automorphism group}
\author{Akito Futaki and Hajime Ono}
\address{Graduate School of Mathematical Sciences, The University of Tokyo, 3-8-1 Komaba Meguro-ku Tokyo 153-8914, Japan}
\email{afutaki@ms.u-tokyo.ac.jp}
\address{Department of Mathematics, Saitama University, 255 Shimo-Okubo, Sakura-Ku,
Saitama 380-8570, Japan}
\email{hono@rimath.saitama-u.ac.jp}

\date{August 9, 2017}

\begin{document}
\begin{abstract} Let $(M,g)$ be a compact K\"ahler manifold and $f$ a positive smooth function such that its
Hamiltonian vector field $K = J\mathrm{grad}_g f$ for the K\"ahler form $\omega_g$ is a holomorphic Killing vector field. 
We say that the pair $(g,f)$ is conformally Einstein-Maxwell
K\"ahler metric if the conformal metric $\Tilde g = f^{-2}g$ has constant scalar curvature. In this paper we prove 
a reductiveness result of the reduced Lie algebra of holomorphic vector fields for conformally Einstein-Maxwell 
K\"ahler manifolds, extending the Lichnerowicz-Matsushima Theorem for constant scalar curvature K\"ahler manifolds. 
More generally we consider extensions of Calabi functional and extremal K\"ahler metrics, and prove an extension 
of Calabi's theorem on the structure of the Lie algebra of holomorphic vector fields
for extremal K\"ahler manifolds. The proof uses a Hessian formula for 
the Calabi functional under the set up of Donaldson-Fujiki picture.
\end{abstract}

\maketitle


\section{Introduction.}

Let $(M,J)$ be a compact complex manifold admitting a K\"ahler structure. A Hermitian metric
$\Tilde{g}$ on $(M,J)$ is said to be 
a conformally K\"ahler, Einstein-Maxwell (cKEM for short) metric
if there exists a positive smooth function $f$ on $M$  such that 
$g=f^2\Tilde{g}$  is K\"ahler, that the Hamiltonian vector field $K=J\mathrm{grad}_gf$
for the K\"ahler form $\omega_g$ 
 is Killing for $g$ (and also for $\Tilde{g}$ necessarily), and that 
the scalar curvature $s_{\Tilde{g}}$ of $\Tilde g$ is constant.

$K$ is necessarily a holomorphic vector field since any Killing vector field on a 
compact K\"ahler manifold is holomorphic. If $f$ is a constant function then $g$ is a K\"ahler metric of constant scalar curvature.
When $f$ is not a constant function, typical known examples are conformally K\"ahler, Einstein metrics
by Page \cite{Page78} on the one-point-blow-up
of $\mathbf C\mathbf P^2$, by Chen-LeBrun-Weber \cite{ChenLeBrunWeber} on the two-point-blow-up
of $\mathbf C\mathbf P^2$, 
by Apostolov-Calderbank-Gauduchon \cite{ACG}, \cite{ACG15} on 4-orbifolds
and by B\'erard-Bergery \cite{BB82} on $\mathbf P^1$-bundles over Fano K\"ahler-Einstein manifolds. 

In the more recent studies, non-Einstein cKEM examples are constructed by LeBrun \cite{L1}, \cite{L2} 
showing that there are ambitoric examples on $\mathbf C\mathbf P^1 \times 
\mathbf C\mathbf P^1$ and the  one-point-blow-up
of $\mathbf C\mathbf P^2$, 
and by Koca-T{\o}nnesen-Friedman \cite{KT} on ruled surfaces of higher genus.
The authors \cite{FO17} also extended LeBrun's construction on $\mathbf C\mathbf P^1 \times 
\mathbf C\mathbf P^1$ to $\mathbf C\mathbf P^1 \times M$ where $M$ is a compact 
constant scalar curvature K\"ahler manifold of arbitrary dimension. 

In this paper we are interested in finding a K\"ahler metric $g$ in a fixed K\"ahler class and a positive
Hamiltonian Killing potential $f$ such that $\Tilde{g} = f^{-2}g$ is a conformally K\"ahler, Einstein-Maxwell metric.
This point of view is taken by Apostolov-Maschler \cite{AM}. 
Thus we may call such a K\"ahler metric $g$ a {\it conformally Einstein-Maxwell K\"ahler metric}. 
Fixing a K\"ahler class, the primary question is which choice of a Killing vector field $K$ is the right one.
The answer is given by the volume minimization as shown by \cite{FO17} in the same spirit as K\"ahler-Ricci solitons
by Tian-Zhu \cite{TZ02} and Sasaki-Einstein metrics by Martelli-Sparks-Yau \cite{MSY2}, see also \cite{FOW}.
The critical points of the volume functional considered in \cite{FO17} are precisely those Killing vector fields 
such that the natural obstruction defined in \cite{AM} vanishes. 
This obstruction will be introduced in Remark \ref{Futaki1} and Remark \ref{Futaki2} in this paper, and is an extension of the one given for constant scalar 
curvature K\"ahler metrics in \cite{futaki83.1}, \cite{futaki83.2}. 
The secondary question is whether the Lichnerowicz-Matsushima theorem (\cite{matsushima57}, \cite{Lic}) asserting
the reductiveness of the Lie algebra of holomorphic vector fields can be extended for conformally
Einstein-Maxwell K\"ahler manifolds. The main theorem of this paper, Theorem \ref{main thm}, gives 
an affirmative answer to this question\footnote{A.Lahdili obtained independently a proof of Theorem \ref{main thm} in \cite{Lahdili17}}.

More generally we consider extensions of Calabi functional and extremal K\"ahler metrics
and prove an extension of Calabi's theorem
for extremal K\"ahler manifolds. 
We call the extended metric an $f$-extremal metric where $f$ stands for the Hamiltonian function of a Killing vector field $K$. The proof uses a Hessian formula for 
the Calabi functional under the set up of Donaldson-Fujiki picture. The arguments in the finite dimensional setting
is given by Wang \cite{Lijing06}. Wang's arguments were effectively employed in \cite{futaki07.1} 
to show an extension of Calabi's
structure theorem for extremal K\"ahler metrics to perturbed extremal K\"ahler metrics defined in \cite{futaki06}.
The proof of the present paper follows closely the lines of arguments in \cite{futaki07.1},
but we will not avoid overlaps 
with \cite{futaki07.1} so as to make the exposition self-contained. 

After this introduction, in section 2, we consider the Donaldson-Fujiki type formulation of the problem, and
compute the first variation of the scalar curvature of the conformal metrics $\Tilde{g}$. This computation and the resulting consequenes
have been obtained by Apostolov-Maschler \cite{AM}. A novelty, if any, is just that the derivation of the formulae 
is closer to the original
style of Calabi's computation \cite{calabi85}, and this enables the computations in section 3 possible. 
In section 3, we prove a Hessian formula of the Calabi functional.
As mentioned above, we follow the derivation of the formula given by Wang \cite{Lijing06} and the first author 
\cite{futaki07.1}. As an example we consider the case of the one-point-blow-up of $\bfC\bfP^2$.
In section4 we give an example of $f$-extremal K\"ahler metrics.


\section{Calabi type formulation and Donaldson-Fujiki set-up.}

Let $\Omega \in H^2_{DR}(M,\bfR)$ be a fixed K\"ahler class, and choose a K\"ahler metric $g$ 
with its K\"ahler form $\omega_g$ in $\Omega$.
Denote by $\mathrm{Aut}(M,\Omega)$ the group of all biholomorphisms of $M$ preserving $\Omega$. Its Lie 
algebra, denoted by $\mathfrak h(M)$, is the same as the Lie algebra of the full automorphism group 
$\mathrm{Aut}(M)$ and consists of all holomorphic vector fields on $M$. Consider the Lie subalgebra
$\mathfrak g$ of $\mathfrak h(M)$ defined as the kernel of the differential of the homomorphism of $\mathrm{Aut}(M)$ to the Albanese
torus induced by the Albanese map. Differential geometric expression of $\mathfrak g$ is given by
\begin{equation}\label{reduced}
\mathfrak g = \{ X \in \mathfrak h(M)\ |\ \mathrm{grad}'u = X\ \text{for\ some\ }u \in C_\bfC^\infty(M)\}
\end{equation}
where $C_\bfC^\infty(M)$ denotes the set of all complex valued smooth functions on $M$ and
\begin{equation}\label{grad}
\mathrm{grad}'u = g^{i\barj}\frac{\partial u}{\partial \barz^j} \frac{\partial}{\partial z^i},
\end{equation}
refer to, for example, \cite{lebrunsimanca93}, \cite{GauduchonLN}.
The Lie subalgebra $\mathfrak g$ is independent of the choice of the K\"ahler metric as can be seen from
the Albanese map description, and is referred to as the
{\it reduced Lie algebra of holomorphic vector fields}. 
For the purpose of the present paper, it is convenient to take (\ref{reduced}) to be the definition of  $\mathfrak g$.
Then the independence of the choice of the metric
can also be seen from the standard fact that a holomorphic vector field $X$ belongs to $\mathfrak g$ if and 
only if $X$ has a zero, see \cite{matsushima71}, \cite{CL77}, \cite{lebrunsimanca93},
\cite{GauduchonLN}. 
If $\omega_g = i g_{i\barj} dz^i \wedge d\barz^j$ is the K\"ahler form then 
$\mathrm{grad}'u = X$ is equivalent to
$$i(X) \omega_g = i \barpartial u.$$
We shall write $X_u:= \mathrm{grad}'u$. 
The Lie algebra structure of $\mathfrak g$ is then given by 
$$ [X_u,Y_w] = X_{\{u,w\}}$$
where $\{u,w\} = \nabla^iu\nabla_i w - \nabla^iw\nabla_iu$, $\nabla^i = g^{i\barj}\nabla_\barj$.

Let $G$ be the connected subgroup of  $\mathrm{Aut}(M)$ corresponding to $\mathfrak g$, i.e.
$\mathrm{Lie}(G) = \mathfrak g$. $G$ is referred to as the {\it reduced automorphism group}.
$G$ is then a subgroup of $\mathrm{Aut}(M,\Omega)$ since $G$ is the connected group containing 
the identity.
Let $T$ be a maximal torus of $G$, and $T^c$ be the complexification
of $T$. Then $T^c$ is a subgroup of the maximal reductive subgroup $G_r$ of $G$.

Let $K \in \mathrm{Lie}(T)$ be arbitrarily chosen. The problem is to find a K\"ahler form $\omega_g \in \Omega$
such that\\
(i) $f$ is a positive smooth function with $J\mathrm{grad}_{g} f = K$, and that\\
(ii) $\Tilde g = f^{-2}g$ is a conformally K\"ahler, Einstein-Maxwell metric.\\
At this point we could assume that $K$ is a critical point of the volume functional introduced in \cite{FO17}
so that the obstruction in \cite{AM} vanishes. But this assumption is needless for the arguments in what follows,
and we do not assume it.

Denote by $\mathrm{Isom}(M,g)$ the group of all isometries of $(M,g)$.
The following is the main theorem of this paper.
\begin{thm}\label{main thm} Let $\Tilde g$ be a conformally K\"ahler, Einstein-Maxwell metric so that (i) and
(ii) above are satisfied. Then 
 the the centralizer 
 $G^K = \{ g \in G\ | \ Ad(g)K = K\}$ of $K$ in  
 the reduced automorphism
 group $G$ is the complexification of $\mathrm{Isom}(M,g) \cap G^K$. In particular $G^K$ is reductive.
\end{thm}

\begin{rem}\label{isometry} 
It is shown in \cite{AM} that, under the conditions (i) and (ii), 
$\mathrm{Isom}(M,\widetilde g) \cap \mathrm{Aut}(M,\Omega)$ is a subgroup of 
$\mathrm{Isom}(M,g) \cap \mathrm{Aut}(M,\Omega)$, and 
the Killing vector field 
$K$ is in the center of the Lie algebra of
$\mathrm{Isom}(M,\widetilde g) \cap \mathrm{Aut}(M,\Omega)$. 
The proof is as follows. 
Since $g = f^2\widetilde{g}$, $\mathrm{Isom}(M,\widetilde g) \cap \mathrm{Aut}(M,\Omega)$ acts as conformal transformations of
$(M,g)$. But, since $g$ is K\"ahler, $\mathrm{Isom}(M,\widetilde g) \cap \mathrm{Aut}(M,\Omega)$ acts on $(M,g)$ as homotheties.
Since $\mathrm{Aut}(M,\Omega)$ preserves $\Omega$, $\mathrm{Isom}(M,\widetilde g) \cap \mathrm{Aut}(M,\Omega)$ acts as isometries of $(M,g)$.
This proves the first statement. Further, again since $g = f^2\widetilde{g}$, $\mathrm{Isom}(M,\widetilde g) \cap \mathrm{Aut}(M,\Omega)$
preserves $f$. This implies $K$ is in the center of $\mathrm{Isom}(M,\widetilde g) \cap \mathrm{Aut}(M,\Omega)$.
\end{rem}

When $f$ is a constant function and $(M,g)$ has a constant scalar curvature, Theorem \ref{main thm} is a part of
Lichnerowicz-Matsushima theorem (\cite{matsushima57}, \cite{Lic}) 
which asserts the reductiveness of the full automorphism
group.

As in the case of the problem finding constant scalar curvature K\"ahler metrics, the Calabi type set-up
fixing an integrable complex structure and varying K\"ahler forms in a fixed K\"ahler class can be turned
through Moser's theorem 
into Donaldson-Fujiki type set-up fixing a symplectic form and varying integrable almost complex structures.

Let $(M, J_0, \omega)$ be a compact K\"ahler manifold, $\mathfrak g$ be the reduced Lie algebra as above,
and $T$ a maximal torus in the reduced automorphism group $G$. 
We fix a Hamiltonian Killing vector field $K \in \mathfrak \mathrm{Lie}(T)$. 
Let $T_K$ be the subtorus obtained as the closure in $T$ of
$\{\exp(tK)\ |\ t\in \bfR \}$. We set
\begin{eqnarray*}
 \mathfrak g^K &:=& \{ X \in \mathfrak g\ |\ [K,X] = 0\}\\
 &=& \{ X \in \mathfrak g\ |\ X = \mathrm{grad}'u \text{\ for\ some\ } u \in C_\bfC^\infty(M),\ Ku = 0\}.
\end{eqnarray*}
Now we consider $\omega$ as a fixed symplectic form on $M$. 
Then we can choose a fixed positive Hamiltonian function $f$ of $K$ by adding a 
positive constant if necessary. 
We may define
$$
C^\infty(M)^K := \{ u \in C^\infty(M)\ |\ Ku = 0\}, 
$$
$$
C_\bfC^\infty(M)^K := \{ u \in C_\bfC^\infty(M)\ |\ Ku = 0\},
$$
$$
C^\infty(M)^K_0 := \{ u \in C^\infty(M)^K\ |\ \int_M u\ f^{-2m-1}\omega^m = 0\}, 
$$
$$
C_\bfC^\infty(M)^K_0 := \{ u \in C_\bfC^\infty(M)^K\ |\ \int_M u\ f^{-2m-1}\ \omega^m = 0\}.
$$
We denote by $\mathrm{Ham}(M)^K$ the group of Hamiltonian diffeomorphisms for 
$C^\infty(M)^K$ or equivalently $C_0^\infty(M)^K$ considered as Hamiltonian functions.
Let $Z^K$ be the space of all $T_K$-invariant $\omega$-compatible integrable almost complex structures $J
\in C^\infty(\mathrm{End}(TM))$. 
Here, $J$ is said to be $\omega$-compatible if the following two conditions are satisfied:\\
(a) $\omega(JX,JY) = \omega(X,Y)$,\\
(b) $g_J := \omega (\cdot,J\cdot)$ is positive definite.\\
By (b),  $g_J$ is a $T_K$-invariant K\"ahler metric for each $J \in Z^K$.

For any fixed $J \in Z^K$ we have the decomposition
$$ T^\ast M \otimes \bfC = T_J^{\ast\prime}M \oplus T_J^{\ast\prime\prime}M$$
into type $(1,0)$-part and type $(0,1)$-part. Naturally $T_J^{\ast\prime\prime}M = \overline{T_J^{\ast\prime}M}$.
For another complex structure $J^\prime \in Z^K$ we also have 
$$ T^\ast M \otimes \bfC = T_{J^\prime}^{\ast\prime}M \oplus T_{J^\prime}^{\ast\prime\prime}M.$$
If $J^\prime$ is close to $J$, $T_{J^\prime}^{\ast\prime}M$ is expressed as a graph over $T_J^{\ast\prime}M$ as
$$ T_{J^\prime}^{\ast\prime}M = \{ \alpha + v(\alpha)\ |\ \alpha \in T_J^{\ast\prime}M \}$$
where 
$$v \in C^\infty(\mathrm{End}(T_J^{\ast\prime}M,T_J^{\ast\prime\prime}M)) \cong C^\infty(T_J^{\prime} M
\otimes T_J^{\ast\prime\prime}M) \cong C^\infty(T_J^{\prime} M
\otimes T_J^{\prime}M)$$
and the last isomorphism uses the K\"ahler metric. As an element of the last term, $v$ is in the 
symmetric part $C^\infty(\mathrm{Sym}(T_J^{\prime} M \otimes T_J^{\prime}M))$, see the proof of Lemma 2.1 in 
\cite{futaki06}.  So, we have
$$ T_J Z^K \subset C^\infty(\mathrm{Sym}(T_J^{\prime} M \otimes T_J^{\prime}M))_\bfR,$$
the last term being the underlying real vector space of $C^\infty(\mathrm{Sym}(T_J^{\prime} M \otimes T_J^{\prime}M))$.
The $L^2$-inner product with respect to the volume form $f^{-2m+1}\omega^m$ gives a K\"ahler structure on $Z^K$.
If $J^\prime = J + \delta J$ and 
$$\alpha + \delta\alpha \in T_{J}^{\ast\prime}M \oplus T_{J}^{\ast\prime\prime}M$$
is in $C^\infty(T_{J^\prime}^{\ast\prime}M)$, then
$$ (J + \delta J)(\alpha + \delta\alpha) = i(\alpha + \delta\alpha),\ \  J\alpha = i\alpha,\ \ J\delta\alpha = -i\delta\alpha.$$
Thus, up to the first order, 
\begin{eqnarray}
(\delta J)\alpha &\equiv& 2i\delta\alpha = -2J\delta\alpha,\\
J^{-1}\delta J \alpha &\equiv& -2\delta\alpha.\label{variation}
\end{eqnarray}
If $J(t)$ is any smooth curve in $Z^K$ for $t \in (-\epsilon,\epsilon)$ with $J(0) = J$, (\ref{variation}) shows that we may regard
$$ J^{-1}\dot{J} \in C^\infty(\mathrm{End}(T_J^{\ast\prime}M,T_J^{\ast\prime\prime}M))_\bfR.$$
Here $\dot{J}$ denotes the derivative of $J(t)$ at $t=0$. 
But (\ref{variation}) also shows
\begin{eqnarray}
J^{-1}\delta J \overline{\alpha } &\equiv& -2\overline{\delta\alpha}.\label{variation2}
\end{eqnarray}
Thus, as a real tensor field, $J^{-1}\dot{J} $ is in the form of 
\begin{eqnarray}
J^{-1}\dot{J} = 2\Re (v^i{}_\bark \frac{\partial}{\partial z^i} \otimes d\barz^k).
\end{eqnarray}
The corresponding curve of K\"ahler metrics $g(t) =
\omega(\cdot,J(t)\cdot)$ may be expressed in the matrix form with respect to local coordinates as
$$ g(t) = \omega J(t).$$
We thus have
\begin{equation}\label{derivative1}
g^{-1}\dot{g} = J^{-1}\omega^{-1}\omega \dot{J} = J^{-1}\dot{J},
\end{equation}
and
\begin{equation}\label{derivative2}
(g^{-1})^\cdot = - g^{-1}\dot{g}g^{-1} = - J^{-1}\dot{J}g^{-1}.
\end{equation}
\begin{lem}\label{derivative3}
Let $\Theta(t) = \barpartial_t (g(t)^{-1}\partial_t g(t))$ be the curvature form of $g(t)$
expressed in terms of time dependent local holomorphic coordinates. Then
\begin{equation}\label{derivative4}
\frac{d}{dt}|_{t=0} \Theta(t) = d^\nabla \nabla (J^{-1}\dot{J}). 
\end{equation}
\end{lem}
\begin{proof}
Consider $\frac{d}{dt}|_{t=0} \Theta(t)$ with normal coordinates with respect to $g(0) = g$ at $p \in M$
so that the derivative 
$dg = 0$ at $p$.
Let $\theta_t = g_t^{-1}\partial_t g_t$ be the connection form. Then at $p$ we have from (\ref{derivative1}) 
$$ \dot\theta = g^{-1}\nabla\dot{g} = \nabla (g^{-1}\dot{g}) = \nabla(J^{-1}\dot{J}).$$
On the other hand, from $\Theta = d\theta + \theta \wedge \theta$, we have
 $$ (d\theta + \theta \wedge \theta)^\cdot = d\dot{\theta} + \dot{\theta}\wedge\theta + \theta\wedge\dot{\theta}
 = d^\nabla \dot{\theta}.$$
 Thus 
 $$\dot{\Theta} = d^{\nabla}\dot{\theta} = d^\nabla \nabla(J^{-1}\dot{J}).$$
\end{proof}

Let $S(J,f)$ be the scalar curvature of the conformal
metric $f^{-2}g_J$. We often write $S_{J,f}$ and $g_{J,f}$ instead of $S(J,f)$ and $f^{-2}g_J$
when these are more convenient. The following theorem is due to Apostolov-Maschler \cite{AM}, 
but we express and prove it in the form which fits to the aim of the present paper.
\begin{thm}[\cite{AM}] \label{AM} For any smooth curve $J(t)$, $-\epsilon < t < \epsilon$, in $Z^K$ with 
\begin{eqnarray}
J^{-1}\dot{J} = 2\Re (v^i{}_\bark \frac{\partial}{\partial z^i} \otimes d\barz^k)  \in T_J Z^K,
\end{eqnarray}
we put $v := v^i{}_\bark\ \frac{\partial}{\partial z^i} \otimes d\barz^k$. Then for any smooth function $h \in C^\infty(M)$
we have
\begin{equation}
\frac{d}{dt}|_{t=0} \int_M S(J(t),f) h f^{-2m-1}\omega^m = 2 \Re\ \int_M (\nabla_i\nabla_j h)\ v^{ij}\ f^{-2m+1}\ \omega^m
\end{equation}
where $v^{ij} = g^{j\bark}v^i{}_\bark$. 
\end{thm}
\begin{proof}
The Ricci form $\mathrm{Ric}_J$ of $g_J$ is given by
$$ \mathrm{Ric}_J = i \mathrm{tr} \Theta_J.$$
Thus the scalar curvature $S_J$ of $g_J$ satisfies
$$ S_J \omega^m = 2\mathrm{Ric}_J \wedge \frac{\omega^{m-1}}{(m-1)!} = 2i\ \mathrm{tr} \Theta_J \wedge \frac{\omega^{m-1}}{(m-1)!}.$$
In Lemma \ref{derivative3}, the $1$-form $\dot{\theta}$ with values in $\mathrm{End}(T^\prime M)$ is written as
$$ \Re (\nabla_j\,v^i{}_\bark\ (dz^j\otimes\frac{\partial}{\partial z^i}) \otimes d\barz^k).$$
Thus we have
$$ \dot{S}_J = 2\Re\ \nabla^\barj\nabla_i v^i{}_\barj = 2\Re\ \nabla_j\nabla_i v^{ij}.$$
The scalar curvature $S_{J,f}$ of $g_{J,f}$ and the scalar curvature $S_J$ of $g_J$ are related by
$$ S_{J,f} = 2\frac{2m-1}{m-1} f^{m+1} \Delta_J (f^{-m+1}) + S_J f^2$$
where $\Delta_J = d_{g_{J}}^\ast d$ is the Hodge Laplacian, see e.g. \cite{B}. Then
\begin{eqnarray}
\int_M S_{J,f} \frac{h}{f^{2m+1}} \omega^m &=& 2m(2m-1)\int_M \frac{h}{f^{2m+1}} g_J^{-1}(df,df)\ \omega^m\nonumber \\
&& - 2(2m-1)\int_M f^{-2m}g_J^{-1}(dh,df) \omega^m \\
&& + \int_M S_J h\, f^{-2m+1} \omega^m.\nonumber 
\end{eqnarray}
Noting 
$$\nabla_i\nabla_j f = \overline{\nabla_\bari\nabla_\barj f} = 0$$ 
since $K$ is a holomorphic vector field, and 
using (\ref{derivative2}) and $v^{ij} = v^{ji}$ we can compute the derivative
\begin{eqnarray}
\frac{d}{dt}|_{t=0}\int_M S_{J,f} \frac{h}{f^{2m+1}} \omega^m &=& -2m(2m-1)\int_M \frac{h}{f^{2m+1}} g^{-1}
({}^t(J^{-1}\dot{J})df,df)\ \omega^m\nonumber \\
&& + 2(2m-1)\int_M f^{-2m}g^{-1}({}^t(J^{-1}\dot{J})dh,df) \omega^m \nonumber \\
&& + \int_M \dot{S}_J h\, f^{-2m+1} \omega^m.\nonumber \\
&=& 2\Re [ -2m(2m-1)\int_M \frac{h}{f^{2m+1}} g^{k\barj}v^i{}_\barj \nabla_i f \nabla_k f \omega^m\nonumber\\
&& + 2(2m-1)\int_M f^{-2m}g^{k\barj}v^i{}_\barj \nabla_i h \nabla_j f \omega^m\nonumber\\
&& + \int_M \nabla^\barj\nabla_i v^i{}_\barj\ h\ f^{-2m+1} \omega^m]\nonumber\\
&=& 2\Re [\int_M \nabla_i\nabla_j h\ v^{ij}\ f^{-2m+1} \omega^m]. \nonumber
\end{eqnarray}
This completes the proof of Theorem \ref{AM}.
\end{proof}

\begin{rem}\label{Futaki1}
Recall we fixed $K \in Lie(T)$, the symplectic form $\omega$ and the positive Hamiltonian function $f$ of $K$
with respect to $\omega$. Let $\mathfrak g_\bfR$ be the subalgebra of $\mathfrak g$ consisting of all $X = \mathrm{grad} u$
with real smooth function $u$. 
The linear map $Fut_f : \mathfrak g_\bfR \to \bfR$ defined by
\begin{equation}
Fut_f (X) = \int_M S(J,f)\, u_X\, \frac{\omega^m}{f^{2m+1}}
\end{equation}
is independent of the choice of $J \in Z$ where $\mathrm{grad}\, u_X = X$ with
\begin{equation}\label{normalization}
\int_M u_X\, f^{-2m-1}\omega^m = 0.
\end{equation}
This follows by taking $h = u_X$ in Theorem \ref{AM}.
\end{rem}

\begin{rem}\label{Futaki2}
By taking $h = 1$ in Theorem $\mathrm{\ref{AM}}$ we see
\begin{equation}
c_{J,f} = \int_M S(J,f)\, f^{-2m-1}\omega^m/\int_M f^{-2m-1}\omega^m
\end{equation}
is independent of $J \in Z$. Thus, we may remove the normalization $\mathrm{(\ref{normalization})}$ and
define $Fut_f$ by
\begin{equation}
Fut_f (X) = \int_M (S(J,f) - c_{J,f})\, u_X\, \frac{\omega_g^m}{f^{2m+1}}.
\end{equation} 
In the Calabi type setup, when $K$ is fixed, $f$ varies as $\omega_g$ varies in $\Omega$ 
in the unique manner with a normalization
\begin{equation}\label{a}
\int_M f\, \omega_g^m = a.
\end{equation}
When this normalization is satisfied we shall write $f$ as $f_{K,g,a}$. 
Then 
$Fut_{K,a} : \mathfrak g_\bfR \to \bfR$ defined by
\begin{equation}
Fut_{K,a} (X) = \int_M (S(\omega_g,f_{K,g,a}) - c_{\Omega,K,a})\, u_X\, \frac{\omega_g^m}{f_{K,g,a}^{2m+1}}
\end{equation} 
with
\begin{equation}\label{eq:3.6}
c_{\Omega,K,a}:=\dfrac
{\displaystyle{\int_M S(\omega_g,f_{K,g,a}) \left(\frac{1}{f_{K,g,a}}\right)^{2m+1}
\frac{\omega^m}{m!}}}
{\displaystyle{\int_M\left(\frac{1}{f_{K,g,a}}\right)^{2m+1}
\frac{\omega^m}{m!}}},
\end{equation}
is independent of the choice of $\omega_g \in \Omega$ where $S(\omega_g,f)$ is the
scalar curvature of $\Tilde g = f^{-2}g$. If there exists a K\"ahler metric $g$ with its K\"ahler form $\omega_g$ 
in $\Omega$ and with $S(\omega_g,f)$
constant then $Fut_{K,a} = 0$. Namely, $Fut_{K,a}$ is an
obstruction to the existence of conformally Einstein-Maxwell K\"ahler metric in the K\"ahler class $\Omega$.
\end{rem}

Since the Hermitian inner product of the tangent space $T_J Z^K$ of $Z^K$ at $J$ is given by the
$L^2$-inner product 
\begin{equation}\label{Hermitian}
(\lambda,\nu)_{L^2(f^{2m+1})} := \int_M \overline{\lambda_{\bari\barj}}\,\nu^{ij}\ f^{-2m+1}\,\omega^m,
\end{equation}
the symplectic structure $\Omega_{J,f}$ on $Z^K$ is given at $J$ by
\begin{eqnarray}
\Omega_{J,f}(\lambda,\nu) &=& \Re (\lambda,\sqrt{-1}\nu) \nonumber \\
&=& 
\Re \int_M \overline{\lambda}_{ij}\,\sqrt{-1}\nu^{ij}\ f^{-2m+1}\,\omega^m. \label{symplectic}
\end{eqnarray}
\begin{prop}[\cite{AM}]\label{AM2}
The $(\cdot,\cdot)_{L^2(f^{-2m-1})}$-dual of the scalar curvature $S(J,f)$ gives an equivariant moment map
for the action of the group $\mathrm{Ham}(M)^K$ of Hamiltonian diffeomorphisms generated by $C_0^\infty(M)^K$ where
$$ (\varphi, \psi)_{L^2(f^{-2m-1})} = \int_M \varphi\,\psi\,f^{-2m-1}\,\omega^m. $$
More precisely, the equality (\ref{equimoment}) below holds and the left hand side is equivariant under
the action on $Z^K$ of the group of Hamiltonian diffeomorphisms generated by $C_0^\infty(M)^K$
\end{prop}
\begin{proof}
If $X = X^\prime + X^{\prime\prime}$ is a smooth complex vector field  where $X^\prime$ and 
$X^{\prime\prime}$ are type $(1,0)$ and $(0,1)$ part respectively, then by Lemma 2.3 in 
\cite{futaki06} 
\begin{equation}\label{tangent}
 L_X J = 2\sqrt{-1}\nabla^{\prime\prime}_J X^\prime - 2 \sqrt{-1}\nabla^\prime_J X^{\prime\prime}.
 \end{equation}
In particular, if $X_u$ is the Hamiltonian vector field of $u \in C^\infty_0(M)^K$, we have
$$ L_{X_u} J = 2\sqrt{-1}\nabla_{\bark}u^i \frac{\partial}{\partial z^i} \otimes d\overline{z^k}  - 
2\sqrt{-1} \nabla_k u^\bari \frac{\partial}{\partial \overline{z^i}} \otimes dz^k. $$
Hence 
$$ \chi_u := 4\Re (\sqrt{-1}\nabla_{\bark}u^i \frac{\partial}{\partial z^i} \otimes d\overline{z^k} ) $$
defines a tangent vector in $T_J Z^K$. For $\alpha = \alpha_i dz^i$, we have
$$ \chi_u(\alpha) = 2\sqrt{-1} (\nabla_\barj u^i) \alpha_i d\overline{z^j}.$$ 
If $J^{-1}\dot{J} = v$ then by Theorem \ref{AM} we have
\begin{eqnarray}
\frac{d}{dt}|_{t=0} \int_M S(J(t),f)\, u\, f^{-2m-1}\omega^m &=& 
2 \Re\ \int_M \overline{u_{\bari\barj}}\ v^{ij}\ f^{-2m+1}\ \omega^m\nonumber\\
&=& \Re \int_M \overline{2\sqrt{-1}\,u^i{}_\barj}\ \sqrt{-1}v_\bari{}^j\ f^{-2m+1}\ \omega^m\nonumber\\
&=& \Omega(\chi_u,v).\label{equimoment}
\end{eqnarray}
Further, since $w \in C_0^\infty(M)^K$ is $T_K$-invariant, the scalar curvature $S(J,f)$ defines a moment map on 
$Z^K$ 
equivariant under the action of Hamiltonian diffeomorphisms generated by $X_w$ for any $w \in C_0^\infty(M)^K$. 
\end{proof}
Note in passing that (\ref{tangent}) shows
\begin{equation}\label{Jtangent1}
 L_{JX_u} J = -2\nabla^{\prime\prime}_J X_u^\prime - 2\nabla^\prime_J X_u^{\prime\prime}.
 \end{equation}
 and the corresponding tangent vector $J^{-1}\dot{J} \in T_J Z^k$ is expressed as
 \begin{equation}\label{Jtangent2}
(J^{-1}\dot{J}) \alpha = -2(\nabla_\barj u^i)\alpha_i d\overline{z^j}
 \end{equation}
 for $\alpha = \alpha_i dz^i$.
 

\section{Hessian formula for the Calabi functional}
Consider the Calabi functional $\Phi : Z^K \to \bfR$ defined by
$$ \Phi(J) = \int_M S_{J,f}^2 \omega^m. $$
If $J$ is a critical point of $\Phi$, the K\"ahler metric $g = \omega J$ is called an $f$-extremal K\"ahler metric.
From Theorem \ref{AM} we see
\begin{equation}\label{first}
\frac{d}{dt} \int_M S_{J(t),f}^2 f^{-2m-1} \omega^m = 8 \Re \int_M \overline{\nabla_\bari\nabla^j S_{J,f}}\,
\sqrt{-1}\nabla^i\nabla_\barj u\, f^{-2m+1} \omega^m
 \end{equation}
 when 
 \begin{eqnarray}
 J^{-1}\dot{J} &=& \chi_u \\
 &=& 2\sqrt{-1}\nabla_{\bark}u^i \frac{\partial}{\partial z^i} \otimes d\overline{z^k}  - 
2\sqrt{-1} \nabla_k u^\bari \frac{\partial}{\partial \overline{z^i}} \otimes dz^k. \nonumber
\end{eqnarray}
Thus we obtain the following.
\begin{lem}\label{extremal} The K\"ahler metric $g = \omega J$ is an $f$-extremal K\"ahler metric if and only if 
$\mathrm{grad}^\prime_J S_{J,f}$ is a holomorphic vector field. 
\end{lem}
We define the fourth order elliptic differential operator $L : C_\bfC^\infty(M) \to C_\bfC^\infty(M)$ by
 \begin{eqnarray}\label{L}
 (w,Lu)_{L^2(f^{-2m-1})} = (\nabla^{\prime\prime}\nabla^{\prime\prime}w, \nabla^{\prime\prime}\nabla^{\prime\prime}u)_{L^2(f^{-2m+1})}.
\end{eqnarray}
We further define the fourth order elliptic differential operator $\barL : C_\bfC^\infty(M) \to C_\bfC^\infty(M)$ by
 \begin{eqnarray}\label{barL}
 \barL u = \overline{L\baru}.
\end{eqnarray}
\begin{lem}\label{derivative5}
If $J^{-1}\dot{J} = 2\Re \nabla^i\nabla_\barj\, u\,\frac{\partial}{\partial z^i} \otimes d\overline{z^j}$ for a real valued smooth
function $u \in C^\infty(M)$, we have
 \begin{eqnarray*}
\frac{d}{dt}|_{t=0}\, S_{J(t),f} = (Lu + \barL u).
\end{eqnarray*}
\end{lem}
\begin{proof} For any real smooth function $w$ we see from Theorem \ref{AM}
\begin{eqnarray*}
&&\frac{d}{dt}|_{t=0}\, \int_M w\,S_{J(t),f}\,f^{-2m-1}\omega^m \\
&& =
\int_M ((\nabla^i\nabla_\barj\, w,\nabla^i\nabla_\barj\, u) + \overline{\nabla^i\nabla_\barj\, w,\nabla^i\nabla_\barj\, u)}
f^{-2m+1}\omega^m\\
&& = (w,Lu) + \overline{(w,Lu)}\\
&& = (w,Lu) + (w,\overline{Lu}) = (w, Lu + \barL u).
\end{eqnarray*}
This completes the proof.
\end{proof}
\begin{lem}\label{equivariance}
For real valued smooth functions $u$ and $w$ in $C^\infty(M)^K$ we have
\begin{eqnarray*}
\Omega(\chi_u,\chi_w) = - \int_M \{w,u\}\,S_{J,f}\,f^{-2m-1}\omega^m.
\end{eqnarray*}
\end{lem}
\begin{proof}
This lemma is simply a restatement of Proposition \ref{AM2}. Let $\sigma$ be in the Hamiltonian diffeomorphisms
generated by the Hamiltonian vector field of $w \in C^\infty(M)^K$. Since $S_{J,f}$ gives a 
$\mathrm{Ham}(M)^K$-equivariant moment map we have
\begin{equation}\label{equivariant2}
\int_M u\,S(\sigma J, f)\,f^{-2m-1}\omega^m = \int_M u\circ\sigma^{-1}\,S_{J,f}\,f^{-2m-1}\omega^m.
\end{equation}
Taking the time differential of $\sigma$ we obtain the lemma by (\ref{equimoment}).
\end{proof}
\begin{lem}\label{Poisson}
For any smooth complex valued smooth function $u \in C^\infty(M)^K$ we have
\begin{equation*}
(\barL - L)u = \frac{i}2\{u,S(J,f)\} = \frac12  (u^\alpha S(J, f)_\alpha - S(J, f)^\alpha u_\alpha)
\end{equation*}
where $u^\alpha = g^{\alpha\barbet}\partial u/\partial \overline{z^\beta}$ for local holomorphic coordinates 
$z^1, \cdots, z^m$.
\end{lem}
\begin{proof}
It is sufficient to prove when $u$ is real valued. For any real valued smooth function $w \in C^\infty(M)^K$ it
follows from Lemma \ref{equivariance} that
\begin{eqnarray*}
&&(w, \barL u - L u)_{L^2(f^{-2m-1}}) \\&&= \overline{(\nabla^{\prime\prime}\nabla^{\prime\prime}w, \nabla^{\prime\prime}\nabla^{\prime\prime}u)_{L^2(f^{-2m+1})} }- (\nabla^{\prime\prime}\nabla^{\prime\prime}w, \nabla^{\prime\prime}\nabla^{\prime\prime}u)_{L^2(f^{-2m+1})}\\
 && = \frac{i}2 \Omega(\chi_v,\chi_u)\\
 && = -\frac{i}2 (\{u,w\}, S_{J,f})_{L^2(f^{-2m-1})}\\
 && = \frac{i}2 (w, \{u,S_{J,f}\})_{L^2(f^{-2m-1})}\\
 && = \frac{i}2 (w, X_u S_{J,f})_{L^2(f^{-2m-1})}\\
 && = \frac{i}2 (w, \omega(X_u, J\grad S_{J,f}))_{L^2(f^{-2m-1})}\\
 && = \frac{i}2 (w, du(J\grad S_{J,f}))_{L^2(f^{-2m-1})}\\
 && = (w, S(J,f)_\alpha u^\alpha  - S(J,f)^\alpha u_\alpha )_{L^2(f^{-2m-1})}.
\end{eqnarray*}
This completes the proof of Lemma \ref{Poisson}.
\end{proof}
\begin{lem}\label{first der}
If $u \in C^\infty(M)^K$ and $J^{-1}\dot{J} = 2\Re \nabla^i\nabla_\barj\, u\,\frac{\partial}{\partial z^i} \otimes d\overline{z^j} 
 \in T_J Z^K$,
then
\begin{eqnarray}
\frac{d}{dt}|_{t=0} \int_M S(J(t),f)^2 f^{-2m-1} \omega^m &=& 
4(u, LS(J,f))_{L^2(f^-{2m-1})}\label{first der2}\\
&=& 4(u, \barL S(J,f))_{L^2(f^-{2m-1})}.\nonumber
\end{eqnarray}
\end{lem}
\begin{proof}
Since $S(J,f) \in C^\infty(M)^K$ we can apply Theorem \ref{AM} to show
that the left hand side of (\ref{first der2}) is equal to 
\begin{eqnarray*}
&&4\Re (\nabla^{\prime\prime}\nabla^{\prime\prime}S(J,f), \nabla^{\prime\prime}\nabla^{\prime\prime}u)_{L^2(f^{-2m+1})}\\
&& = 2((\nabla^{\prime\prime}\nabla^{\prime\prime}S(J,f), \nabla^{\prime\prime}\nabla^{\prime\prime}u)_{L^2(f^{-2m+1})} + 
(\nabla^{\prime\prime}\nabla^{\prime\prime}u, \nabla^{\prime\prime}\nabla^{\prime\prime}S(J,f))_{L^2(f^{-2m+1})})\\
&& = 2(u,LS(J,f))_{L^2(f^{-2m-1})} + 2(u,\barL S(J,f))_{L^2(f^{-2m-1})}.
\end{eqnarray*}
But Lemma \ref{Poisson} implies 
$$ \barL S(J,f) = LS(J,f).$$
Hence the left hand side of (\ref{first der2}) is equal to
$$ 4(u,LS(J,f))_{L^2(f^{-2m-1})} = 4(u,\barL S(J,f))_{L^2(f^{-2m-1})}.$$
\end{proof}
\begin{lem}\label{derL}
Suppose that $(\omega,J,f)$ is an $f$-extremal K\"ahler metric so that $J\grad S(J,f)$ is a holomorphic vector field.
If $J^{-1}\dot{J} = 2\Re \nabla^i\nabla_\barj\, u\,\frac{\partial}{\partial z^i} \otimes d\overline{z^j}$ for some real smooth function $u \in C^\infty(M)^K$ then we have
\begin{eqnarray*}
(\frac{d}{dt}|_{t=0} L)S(J,f) &=& - \frac12 L(S(J,f)^\alpha u_{\alpha} - u^\alpha S(J,f)_\alpha) \\
&=& L(\barL - L)u
\end{eqnarray*}
\end{lem}
\begin{proof}
First note that if 
$$ i(X_u) \omega = du$$
then 
\begin{equation}\label{JX_u}
L_{\frac12 JX_u} \omega = i\partial\barpartial u.
\end{equation}
Let $\{f_s\}$ be the flow generated by $-\frac12 JX_u$. Let $S$ be a smooth function on $M$ such that $\grad\, S$ 
is a holomorphic vector field. We shall compute 
$ \frac{d}{ds}|_{s=0} L(f_sJ,\omega) S$, and apply to $S = S(J,f)$, and obtain the conclusion of
Lemma \ref{derL}. Let $\{S_s\}$ be a family of smooth functions such that
$S_0 = S$, that
$$ \grad_s^\prime\,S_s = \grad^\prime\, S,$$
where $\grad_s$ denotes the gradient with respect to $f_{-s}^\ast \omega$, and that
$$ \int_M S_s (f_{-s}^\ast \omega)^m = \int_M S \omega^m. $$
This implies
\begin{equation}\label{S_s1}
L(f_sJ,\omega)f_s^\ast S_s = f_s^\ast(L(J,f_{-s}^\ast \omega) S_s) = 0.
\end{equation}
On the other hand, 
in general, if $f_{-s}^\ast \omega = \omega + i\partial\barpartial \varphi$ then $S_s = S + S^\alpha\varphi_\alpha$.
Therefore, since $L_{\frac12 JX_u} \omega = i\partial\barpartial u$ by (\ref{JX_u}) we have
\begin{equation}\label{S_s1}
S_s = S + s S^\alpha\,u_\alpha + O(s^2).
\end{equation}
Thus taking the derivative of (\ref{S_s1}), we obtain
\begin{equation}\label{derL2}
(\frac{d}{ds}|_{s=0}L)S + L(-\frac12 (JX_u)S + S^\alpha u_\alpha) = 0.
\end{equation}
By an elementary computation we see
\begin{eqnarray*}
(JX_u)S &=& g(JX_u, \grad\,S) = \omega(X_u, \grad\, S) = du(\grad\, S)\\
&=& (\partial u + \barpartial u)(\nabla^\prime S + \nabla^{\prime\prime} S) = u_\alpha S^\alpha + u^\alpha S_\alpha.
\end{eqnarray*}
Thus, from (\ref{derL2}) and the above computation, we obtain
\begin{eqnarray*}
(\frac{d}{ds}|_{s=0}L)S &=& L(\frac12 (u_\alpha S^\alpha + u^\alpha S_\alpha) - S^\alpha u_\alpha) = 0\\
&=& \frac12 L (u^\alpha S_{\alpha} - u_\alpha S^\alpha)\\
&=& L(\barL - L) u.
\end{eqnarray*}
This completes the proof of Lemma \ref{derL}.
\end{proof}
To express the Hessian formula, for a real smooth function $u \in C^\infty(M)^K$, we identify
$J^{-1}\dot{J} = 2\Re \nabla^i\nabla_\barj\, u\,\frac{\partial}{\partial z^i} \otimes d\overline{z^j}$  with 
$\nabla^{\prime\prime}\nabla^{\prime\prime} u$.
\begin{thm}\label{Hessian}
Let $J$ be a critical point of $\Phi$ so that $(\omega,J,f)$ is an $f$-extremal K\"ahler metric.
Let $u$ and $w$ be real smooth functions in $C^\infty(M)^K$. Then the Hessian $\mathrm{Hess}(\Phi)_J$ at
$J$ is given by
$$\mathrm{Hess}(\Phi)_J(\nabla^{\prime\prime}\nabla^{\prime\prime} u, \nabla^{\prime\prime}\nabla^{\prime\prime} w)
= 8(u, L\barL w) = 8(u,\barL L w). $$
In particular, $L\barL = \barL L$ on $C_\bfC^\infty(M)^K$ at any critical point of $\Phi$.
\end{thm}
\begin{proof}
Suppose $J^{-1}\dot{J} = 2\Re \nabla^i\nabla_\barj\, w\,\frac{\partial}{\partial z^i} \otimes d\overline{z^j}$, or
$\nabla^{\prime\prime}\nabla^{\prime\prime} w$ by our identification. Then by Lemma \ref{first der}, Lemma \ref{derL} 
and Lemma \ref{derivative5} we obtain
\begin{eqnarray*}
\mathrm{Hess}(\Phi)_J(\nabla^{\prime\prime}\nabla^{\prime\prime} u, \nabla^{\prime\prime}\nabla^{\prime\prime} w) &=& 
\frac{d}{dt}|_{t=0} 4(u,LS(J,f))_{L^2(f^{-2m-1})}\\
&=& 4(u, (\frac{d}{dt}|_{t=0} L)S(J,f) + L\frac{d}{dt}|_{t=0} S(J(t),f))_{L^2(f^{-2m-1})}\\
&=& 4(u, L(\barL - L)w + L(L + \barL)w)_{L^2(f^{-2m-1})}\\
&=& 8(u,L\barL w)_{L^2(f^{-2m-1})}.
\end{eqnarray*}
Similarly, we obtain
\begin{eqnarray*}
\mathrm{Hess}(\Phi)_J(\nabla^{\prime\prime}\nabla^{\prime\prime} u, \nabla^{\prime\prime}\nabla^{\prime\prime} w) &=& 
\frac{d}{dt}|_{t=0} 4(u,\barL S(J,f))_{L^2(f^{-2m-1})}\\
&=& 4(u, (\frac{d}{dt}|_{t=0} \barL)S(J,f) + \barL\frac{d}{dt}|_{t=0} S(J(t),f))_{L^2(f^{-2m-1})}\\
&=& 4(u, \barL(L - \barL)w + \barL(L + \barL)w)_{L^2(f^{-2m-1})}\\
&=& 8(u,\barL L w)_{L^2(f^{-2m-1})}.
\end{eqnarray*}
This completes the proof of Theorem \ref{Hessian}.
\end{proof}
The following theorem extends a theorem of Calabi \cite{calabi85} for extremal K\"ahler metrics.
\begin{thm}\label{Calabi}
If $g = \omega J$ is an $f$-extremal K\"ahler metric 
with $K = J\grad f$ then the centralizer $\mathfrak g^K$ of $K$ in the reduced Lie algebra $\mathfrak g$ of holomorphic vector fields
has the following structure:
\begin{itemize}
\item[(a)] $\mathfrak g_0^K := (\mathfrak i(M) \cap \mathfrak g^K)\otimes \bfC$ is the maximal reductive subalgebra of
$\mathfrak g^K$ where $\mathfrak i(M)$ denotes the real Lie algebra of all Killing vector fields.
\item[(b)] $\grad^\prime S_{J,f} = g^{i\barj}\frac{\partial S_{J,f}}{\partial \overline{z^j}}$ is in the center of $\mathfrak g_0^K$.
\item[(c)] $\mathfrak g^K = \mathfrak g_0^K + \sum_{\lambda \ne 0} \mathfrak g_\lambda^K$ where 
$\mathfrak g_\lambda^K$ is the $\lambda$-eigenspace of $\mathrm{ad}(\grad^\prime S_{J,f})$. Moreover, we have
$[\mathfrak g_\lambda^K,\mathfrak g_\mu^K] \subset \mathfrak g_{\lambda + \mu}^K$. 
\end{itemize}
\end{thm}
\begin{proof}
By Theorem \ref{Hessian}, $L\barL = \barL L$ on $C^\infty_\bfC(M)^K$. Therefore $\barL$ maps $\ker L$ to $\ker L$,
and we have the direct sum decomposition 
$$ \ker L = \sum_\lambda E_\lambda $$
into the eigenspaces of $2\barL$. Further by Lemma \ref{Poisson}
\begin{eqnarray*}
\lambda u &=& 2\barL u \\
&=& 2(\barL - L) u\\
&=& S(J,f)^\alpha u_\alpha - u^\alpha S(J,f)_\alpha.
\end{eqnarray*}
This shows
$$ [\grad^\prime S(J,f), \grad^\prime u] = \lambda \grad^\prime u.$$
Thus we have $E_\lambda = \mathfrak g_\lambda^K$. For $\lambda = 0$ we have
$$ \mathfrak g_0^K = E_0 = \ker L \cap \ker \barL.$$
Since $\barL u = 0$ is equivalent to $L\baru = 0$, if $u \in \mathfrak g_0^K$ then
$u$ satisfies both $L u = 0$ and $L\baru = 0$. This implies $L\Re u = 0$ and $L\Im u = 0$.
In general if $\grad^\prime u$ is a holomorphic vector field for a real smooth function $u$
then $J\grad\, u$ is a Killing vector field. Hence we obtain
$$ \mathfrak g_0^K = (\mathfrak i(M) \cap \mathfrak g_0^K)\otimes \bfC.$$
\end{proof}
Now we are in a position to prove Theorem \ref{main thm}
\begin{proof}[Proof of Theorem \ref{main thm}]
If $g = \omega J$ is a conformally Einstein-Maxwell K\"ahler metric, then $g$ is an $f$-extremal K\"ahler
metric with $S(J,f)$ is a constant function. Therefore, in this case 
$$\mathfrak g^K = \mathfrak g_0^K = (\mathfrak i(M) \cap \mathfrak g^K)\otimes \bfC.$$
Since $\mathrm{Isom}(M,g)$ is compact, $\mathfrak g^K$ is reductive.
This completes the proof of Theorem \ref{main thm}.
\end{proof}

\begin{ex}
We consider the case of the one-point-blow-up 
$\widehat{\boldsymbol{\mathbf{C}\mathbf P^2}}$ of $\boldsymbol{\mathbf{C}\mathbf P^2}$.
Let 
$\Delta_p$ be the convex hull of $(0,0),(p,0),(p,1-p),(0,1),\ (0<p<1)$ in $(\mu_1, \mu_2)$-plane.
Then for each $p$, $\Delta_p$ determines a K\"ahler class of $\widehat{\boldsymbol{\mathbf{C}\mathbf P^2}}$. The Hamiltonian function $f$ of a holomorphic Killing vector field is 
an affine linear function 
$f = a\mu_1+b\mu_2+c$, which is determined uniquely up to the choice of $c$.
Then $f$ is positive on $\Delta_p$
if and only if 
$$c,b+c,(1-p)b+pa+c,pa+c>0.$$ 
In \cite{FO17} we showed that the obstruction $Fut_f$ in Remark \ref{Futaki1} vanishes if and only if $K = J\mathrm{grad} f$ gives a critical point of the volume functional defined in Theorem 1.1 in \cite{FO17}. The Hamiltonian functions $f$ which correspond to critical points are determined in section 4.3 in \cite{FO17}. The following is the list of critical points.

\begin{enumerate}
\item[(1)]\ $a=\frac{p+2\sqrt{1-p}-2}{2{p}^{2}},b=0, 0 < p < 1$. 
\item[(2)]\ $a=-\frac{\sqrt{9{p}^{2}-8p}+p}{4{p}^{2}},b=0$. $\frac 89 < p < 1$. 
\item[(3)]\ $a=\frac{\sqrt{9{p}^{2}-8p}-p}{4{p}^{2}},b=0$. $\frac 89 < p < 1$.
\item[(4)]\ $a=-\frac{\sqrt{{p}^{4}-4{p}^{3}+16{p}^{2}-16p+4}-{p}^{2}+4p-2}{2{p}^{3}-4{p}^{2}+12p-8},
b=-\frac{\sqrt{{p}^{4}-4{p}^{3}+16{p}^{2}-16p+4}}{{p}^{3}-2{p}^{2}+6p-4}$.\\
$0 < p < \alpha$. 
\item[(5)]\ $a=\frac{\sqrt{{p}^{4}-4{p}^{3}+16{p}^{2}-16p+4}+{p}^{2}-4p+2}{2{p}^{3}-4{p}^{2}+12p-8},
b=\frac{\sqrt{{p}^{4}-4{p}^{3}+16{p}^{2}-16p+4}}{{p}^{3}-2{p}^{2}+6p-4}$.\\
$0 < p < \alpha$.  
\item[(6)]\ $a=\frac{2\sqrt{-9{b}^{2}{p}^{3}+\left( 21{b}^{2}+1\right) {p}^{2}+\left( 1-16{b}^{2}\right) p+4{b}^{2}-1}+3b{p}^{2}+\left( 1-2b\right) p}{6{p}^{2}-4p}$.
\item[(7)]\ $a=-\frac{2\sqrt{-9{b}^{2}{p}^{3}+\left( 21{b}^{2}+1\right) {p}^{2}+\left( 1-16{b}^{2}\right) p+4{b}^{2}-1}-3b{p}^{2}+\left( 2b-1\right) p}{6{p}^{2}-4p}$.
\end{enumerate}
Here, $\alpha$ is the smallest positive root of $p^4-4p^3+16p^2-16p+4=0$.
In the cases (1), (2) and (3) we have $b = 0$ which shows the solution has to have $U(2)$-symmetry.
In fact LeBrun \cite{L2} constructed a solution in each cases. $K$ and the centralizer $G^K$ are then
of the form
$$
 \left(\begin{array}{ccc}
\alpha & 0 & 0 \\
0 & \beta & 0 \\
0 & 0 & \beta
\end{array}\right), \qquad
\left\{ \left(\begin{array}{ccc}
\ast & 0 & 0 \\
0 & \ast & \ast \\
0 & \ast & \ast
\end{array}\right)\right\} $$
in $PGL(3, \bfC)$ where $\alpha \ne \beta$, $\alpha \ne 0$ and $\beta \ne 0$.
In the cases (4), (5), (6) and (7), the existence is not known at this moment of writing,
but if a solution exists it must have $U(1) \times U(1)$-symmetry. 
$K$ and the centralizer $G^K$ are then
of the form
$$
 \left(\begin{array}{ccc}
\alpha & 0 & 0 \\
0 & \beta & 0 \\
0 & 0 & \gamma
\end{array}\right), \qquad
\left\{ \left(\begin{array}{ccc}
\ast & 0 & 0 \\
0 & \ast & 0\\
0 & 0 & \ast
\end{array}\right) \right\}$$
in $PGL(3, \bfC)$ where $\alpha$, $\beta$ and $\gamma$ are non-zero and mutually distinct.
\end{ex}


\section{Construction of $f$-extremal K\"ahler metrics}

In this section we give a construction of $f$-extremal K\"ahler metrics 
on $\bfC\bfP^1 \times M$ when $M$ is an $(m-1)$-dimensional compact complex manifold with a
K\"ahler metric $g_2$ of 
constant scalar curvature $s_{g_2}=c$. This is an extension of a construction of conformally K\"ahler
Einstein-Maxwell metrics given in section 3 of \cite{FO17}.

Let $g_1$ be an $S^1$-invariant metric on $\mathbf{C}\mathbf P^1$.
Using the action-angle coordinates
$(t,\theta)\in (a,b)\times (0,2\pi]$, the $S^1$-invariant metric $g_1$ can be written as
$$
g_1=\dfrac{dt^2}{\Psi(t)}+\Psi(t)d\theta^2
$$
for some smooth function $\Psi (t)$ where the Hamiltonian function of the generator
of the $S^1$-action is $t$. Therefore $f = t$ in this case.
We put $g=g_1+g_2$.
We wish to construct $\Psi$ such that the gradient vector field with respect to $g$ of the scalar curvature $s(\Tilde g)$ of the
Hermitian metric $\Tilde g = g/t^2$
on $\mathbf{C}\mathbf P^1\times M$ is a holomorphic vector field.
This happens to be the case if 
\begin{equation}\label{4.1}
s(\Tilde g) = dt + e
\end{equation}
for some constants $d$ and $e$.

Since the scalar curvature of $g_1$ is given by
$$
s_1=\Delta_{g_1}\log \Psi=-\Psi''(t),
$$
the scalar curvature of $g$ is given by
$$
s=s_1+s_2=c-\Psi''(t).
$$
It follows that the equation (\ref{4.1}) is equivalent to 
\begin{equation}\label{eq:5.1}
dt + e =2\left(\dfrac{2m-1}{m-1}\right)t^{m+1}\Delta_g(t^{1-m})+
(c-\Psi''(t))t^2.
\end{equation}
Using
$$
\Delta_{g_1}\left(\dfrac{1}{t^{m-1}}\right)=(m-1)\left(
\dfrac{\Psi}{t^m}
\right)',
$$
the equation \eqref{eq:5.1} reduces to the ODE
\begin{equation}\label{eq:5.2}
t^2\Psi''-2(2m-1)t\Psi'+2m(2m-1)\Psi=ct^2-dt-e.
\end{equation}
The general solution of the equation \eqref{eq:5.2} is
\begin{equation}\label{eq:5.3}
\Psi=At^{2m}+Bt^{2m-1}+
\dfrac{c}{2(m-1)(2m-3)}t^2-\dfrac{d}{2(m-1)(2m-1)}t-\dfrac{e}{2m(2m-1)}
\end{equation}
with the requirement of 
$$\Psi(t)>0$$
 on $(a,b)$. 
The boundary conditions are
$$
\Psi(a)=\Psi(b)=0,\Psi'(a)=-\Psi'(b)=2,
$$
which reduce to a simultaneous linear equation for $A,B,c, d$ and $e$.
The space of solutions are $1$-dimensional. If we express it in terms $B$ we have
the following expression. \\
(i)\ \ $A$ is given by
$$
A=-\frac{A_1 B}{-2 a{b}^{2m-1}\ m+2\ {b}^{2m}\ m+2{a}^{2m-1}\ bm-2 {a}^{2m}m-2 {b}^{2m}+2{a}^{2m}}
$$
with
\begin{eqnarray*}
A_1 &=&  2{b}^{2m-1}m-2a{b}^{2m-2}m+2{a}^{2m-2}bm-2 {a}^{2m-1}m-3 {b}^{2m-1}\\
&& +a{b}^{2m-2} -{a}^{2m-2} b+3{a}^{2m-1}.
\end{eqnarray*}
(ii)\ \ $c$ is given by
\begin{equation*}
c= \frac {(m-1)(2m-3) P}Q B -\frac{4(m-1)(2m-3)}{b-a}
\end{equation*}
with
\begin{eqnarray*}
P &=& ( 2{a}^{m}{b}^{m+1}m-2{a}^{m+1}{b}^{m}m-{a}^{m}{b}^{m+1}-a{b}^{2m}+{a}^{m+1}{b}^{m}+{a}^{2m}b)\\
&&\cdot( 2{a}^{m}{b}^{m+1}m-2{a}^{m+1}{b}^{m}m-{a}^{m}{b}^{m+1}+a{b}^{2m}+{a}^{m+1}{b}^{m}-{a}^{2m}b),
\end{eqnarray*}
\begin{eqnarray*}
Q &=& ab( a{b}^{2m+2}m-2{a}^{2}{b}^{2m+1}m+{a}^{3}{b}^{2m}m+{a}^{2m}{b}^{3}m-2{a}^{2m+1}{b}^{2}m\\
&& +{a}^{2m+2}bm-a{b}^{2m+2}+{a}^{2}{b}^{2m+1}+{a}^{2m+1}{b}^{2}-{a}^{2m+2}b).
\end{eqnarray*}
(iii) $d$ is given by 
\begin{equation*}
d = \frac{(2(m-1)(2m-1)R}S B -\frac{4(a+b)(m-1)(2m-1)}{b-a}
\end{equation*}
with
\begin{eqnarray*}
R &=& 2{a}^{2m}{b}^{2m+3}{m}^{2}-2{a}^{2m+1}{b}^{2m+2}{m}^{2}-2{a}^{2m+2}{b}^{2m+1}{m}^{2}+2{a}^{2m+3}{b}^{2m}{m}^{2}\\
&&-3{a}^{2m}{b}^{2m+3}m+3{a}^{2m+1}{b}^{2m+2}m+3{a}^{2m+2}{b}^{2m+1}m\\
&&-3{a}^{2m+3}{b}^{2m}m+{a}^{2m}{b}^{2m+3}-{a}^{3}{b}^{4m}+{a}^{2m+3}{b}^{2m}-{a}^{4m}{b}^{3}),
\end{eqnarray*}
\begin{eqnarray*}
S &=& ab
(a{b}^{2m+2}m-2{a}^{2}{b}^{2m+1}m+{a}^{3}{b}^{2m}m+{a}^{2m}{b}^{3}m-2{a}^{2m+1}{b}^{2}m+{a}^{2m+2}bm\\
&&-a{b}^{2m+2}+{a}^{2}{b}^{2m+1}+{a}^{2m+1}{b}^{2}-{a}^{2m+2}b).
\end{eqnarray*}
(iv) $e$ is given by
\begin{equation*}
e = -\frac{m(2m-1)T}U B + \frac{4abm(2m-1)}{b-a}
\end{equation*}
with
\begin{eqnarray*}
T &=& 4{a}^{2m+1}{b}^{2m+3}{m}^{2}-8{a}^{2m+2}{b}^{2m+2}{m}^{2}+4{a}^{2m+3}{b}^{2m+1}{m}^{2} -8{a}^{2m+1}{b}^{2m+3}m\\
&&+16{a}^{2m+2}{b}^{2m+2}m-8{a}^{2m+3}{b}^{2m+1}m
+4{a}^{2m+1}{b}^{2m+3}-6{a}^{2m+2}{b}^{2m+2}\\
&&+4{a}^{2m+3}{b}^{2m+1}-{a}^{4}{b}^{4m}-{a}^{4m}{b}^{4},
\end{eqnarray*}
\begin{eqnarray*}
U &=& ab(a{b}^{2m+2}m-2{a}^{2}{b}^{2m+1}m+{a}^{3}{b}^{2m}m+{a}^{2m}{b}^{3}m-2{a}^{2m+1}{b}^{2}m+{a}^{2m+2}bm\\
&&-a{b}^{2m+2}+{a}^{2}{b}^{2m+1}+{a}^{2m+1}{b}^{2}-{a}^{2m+2}b).
\end{eqnarray*}
We used Maxima to obtain the above result. For example, if we put $B=0$ then we obtain
\begin{equation}\label{3}
\begin{split}
A&=0,\\
c&=-\frac{4(m-1)(2m-3)}{b-a},\\
d&=-\frac{4(a+b)(m-1)(2m-1)}{b-a},\\
e&=\frac{4abm(2m-1)}{b-a}
\end{split}
\end{equation}
and 
$$
\Psi(t)=-\frac{2(t-a)(t-b)}{b-a}.
$$
This $\Psi$ is positive on $(a,b)$ and satsifies the boundary conditions.

As another example, we set $a=1,b=2$ for simplicity, and put
\begin{equation}\label{4}
\begin{split}
A &= \frac{2\left( {2}^{2m+1}m-5\cdot
{2}^{2m}+8m+4\right) }{\left( -{2}^{m+1}m+{2}^{2m}+{2}^{m}-2\right) \left( {2}^{m+1}m+{2}^{2m}-{2}^{m}-2\right) },\\
B &=
\frac{-8\left( {2}^{2m}m-{2}^{2m+1}+2m+2\right) }{\left(- {2}^{m+1}m+{2}^{2m}+{2}^{m}-2\right) \left( {2}^{m+1}m+{2}^{2m}-{2}^{m}-2\right) },\\
c &=  0,\\
d &=
-\frac{4\left( m-1\right) \left( 2m-1\right) \left( -3m{2}^{2m+1}+{2}^{4m}+3\cdot{2}^{2m}-4\right) }{\left( -{2}^{m+1}m+{2}^{2m}+{2}^{m}-2\right) \left( {2}^{m+1}m+{2}^{2m}-{2}^{m}-2\right) },\\
e&=
\frac{4m\left( 2m-1\right) \left( -{2}^{2m+3}m+3\cdot{2}^{2m+1}+{2}^{4m}-8\right) }{\left( -{2}^{m+1}m+{2}^{2m}+{2}^{m}-2\right) \left( {2}^{m+1}m+{2}^{2m}-{2}^{m}-2\right) }.
\end{split}
\end{equation}
Then $\Psi$ satisfies the boundary conditions. Further $\Psi > 0$ on the interval $(1,2)$ 
because of  $Ad<0$
and the following Lemma.
\begin{lem}\label{positive}
Let $m$ be an integer greater than $1$, and suppose $0<a<b$. 
If the real valued function
$$
f(t)=\alpha t^{2m}+\beta t^{2m-1}+\gamma t^2+\delta t+\varepsilon
$$
with $\alpha \delta >0$ 
satisfies the boundary conditions
$$
f(a)=f(b)=0, \ f'(a)>0, -f'(b)>0
$$
then $f$ is positive on the interval $(a,b)$.
\end{lem}
\begin{proof}
Suppose that there is a $c\in (a, b)$ such that $f(c)\le 0$. Then by the boundary conditions
$f$ has at least three critical points on $(a, b)$. On the other hand, we have 
$$
\dfrac{f'(t)}{t}=2m\alpha t^{2m-2}+(2m-1)\beta t^{2m-3}+2\gamma +
\dfrac{\delta}{t},
$$
and thus
\begin{equation}\label{prime}
\left(\dfrac{f'(t)}{t}\right)'=t^{2m-4}\{2m(2m-2)\alpha t+(2m-1)(2m-3)\beta
-\dfrac{\delta}{t^{2m-2}}\}.
\end{equation}
From $\alpha\delta>0$, the right hand side of (\ref{prime}) changes sign only once
on the region $t>0$. This implies $f'/t$ has at most two zeros on the interval
$(a,b)$. This is a contradiction, and completes the proof of Lemma \ref{positive}.
\end{proof}

\end{document}